\documentclass[11pt,letterpaper]{article}

\usepackage{graphicx}
\usepackage{url}
\usepackage{amsfonts}
\usepackage{amsmath}
\usepackage{amsthm}
\usepackage{amssymb}
\usepackage{url}
\usepackage{subfig}
\usepackage[colorlinks=true,citecolor=black,linkcolor=black,urlcolor=blue]{hyperref}
\usepackage{fullpage}
\usepackage{xcolor}
\usepackage[mathlines]{lineno}

\newcommand*\patchAmsMathEnvironmentForLineno[1]{%
  \expandafter\let\csname old#1\expandafter\endcsname\csname #1\endcsname
  \expandafter\let\csname oldend#1\expandafter\endcsname\csname end#1\endcsname
  \renewenvironment{#1}%
     {\linenomath\csname old#1\endcsname}%
     {\csname oldend#1\endcsname\endlinenomath}}%
\newcommand*\patchBothAmsMathEnvironmentsForLineno[1]{%
  \patchAmsMathEnvironmentForLineno{#1}%
  \patchAmsMathEnvironmentForLineno{#1*}}%
\AtBeginDocument{%
\patchBothAmsMathEnvironmentsForLineno{equation}%
\patchBothAmsMathEnvironmentsForLineno{align}%
\patchBothAmsMathEnvironmentsForLineno{flalign}%
\patchBothAmsMathEnvironmentsForLineno{alignat}%
\patchBothAmsMathEnvironmentsForLineno{gather}%
\patchBothAmsMathEnvironmentsForLineno{multline}%
}

\allowdisplaybreaks

\newtheorem{theorem}{Theorem}[section]

\newtheorem{lemma}[theorem]{Lemma}

\newcommand{\M}{\ensuremath{\mathcal{M}}}
\newcommand{\E}{\ensuremath{\mathcal{E}}}
\newcommand{\D}{\ensuremath{\mathcal{B}}}

\title{Center of maximum-sum matchings of bichromatic points}

\author{
	Pablo P\'erez-Lantero\thanks{Universidad de Santiago de Chile (USACH), Facultad de Ciencia, Departamento de Matem\'atica y Ciencia de la Computaci\'on, Chile {\tt pablo.perez.l@usach.cl}.}
	\and
	Carlos Seara\thanks{Departament de Matem\`{a}tiques,
		Universitat Polit\`{e}cnica de Catalunya, Spain.
		{\tt carlos.seara@upc.edu}.}
}

\begin{document}
\maketitle

%\blfootnote{
%	\begin{minipage}[l]{0.3\textwidth}
%		\includegraphics[trim=10cm 6cm 10cm 5cm,clip,scale=0.15]{eu_logo}
%	\end{minipage}
%	\hspace{-3cm}
%	\begin{minipage}[l][1cm]{0.75\textwidth}	
%       	This work has received funding from the European Union's Horizon 2020
%       	research and innovation programme under the Marie Sk\l{}odowska-Curie
%       	grant agreement No 734922.
%     \end{minipage}}
%
%\vspace{-0.9cm}

\begin{abstract}
	Let $R$ and $B$ be two disjoint point sets in the plane with $|R|=|B|=n$. Let $\M=\{(r_i,b_i),i=1,2,\ldots,n\}$ be a perfect matching that matches points of $R$ with points of $B$ and maximizes $\sum_{i=1}^n\|r_i-b_i\|$, the total Euclidean distance of the matched pairs. In this paper, we prove that there exists a point $o$ of the plane (the center of $\M$) such that $\|r_i-o\|+\|b_i-o\|\le \sqrt{2}~\|r_i-b_i\|$ for all $i\in\{1,2,\ldots,n\}$.
\end{abstract}

\section{Introduction}

Let $R$ and $B$ be two disjoint point sets in the plane with $|R|=|B|=n$, $n\ge 1$. The points in $R$ are \emph{red}, and those in $B$ are \emph{blue}. A \emph{matching} of $R\cup B$ is a partition of $R\cup B$ into $n$ pairs such that each pair consists of a red and a blue point. A point $p\in R$ and a point $q\in B$ are \emph{matched} if and only if the (unordered) pair $(p,q)$ is in the matching. For every $p,q\in\mathbb{R}^2$, we use $pq$ to denote the segment connecting $p$ and $q$, and $\|p-q\|$ to denote its length, which is the Euclidean norm of the vector $p-q$. Let $\D(pq)$ denote the disk with diameter equal to $\|p-q\|$, that is centered at the midpoint $\frac{p+q}{2}$ of the segment $pq$. For any matching $\M$, we use $\D_\M$ to denote the set of the disks associated with the matching, that is, $\D_\M=\{\D(pq):(p,q) \in\M\}$.

In this note, we consider the \emph{max-sum} matching $\M$, as the matching that maximizes the total Euclidean distance of the matched points. As our main result, we prove the following theorem:

\begin{theorem}\label{theo:main}
	There exists a point $o$ of the plane such that for all $i\in\{1,2,\ldots,n\}$ we have:
	\[
		\|r_i-o\|+\|b_i-o\|~\le~ \sqrt{2}\, \|r_i-b_i\|.
	\]
\end{theorem}

Fingerhut (see Eppstein~\cite{andy}), motivated by a problem in designing communication networks (see Fingerhut et al.~\cite{FingerhutST97}), conjectured that given a set $P$ of $2n$ uncolored points in the plane and a max-sum matching $\{(a_i,b_i),i=1,\dots,n\}$ of $P$, there exists a point $o$ of the plane, not necessarily a point of $P$, such that
\begin{equation}\label{eqFingerhut}
    \|a_i-o\|+\|b_i-o\|~\le~ \frac{2}{\sqrt{3}}~\|a_i-b_i\| \hspace{0.3cm} \textrm{for all } i\in\{1,\ldots,n\}, ~\textrm{where } 2/\sqrt{3}\approx 1.1547.
\end{equation}
Bereg et al.~\cite{bereg2022} obtained an approximation to this conjecture. They proved that for any point set $P$ of $2n$ uncolored points in the plane and a max-sum matching $\M=\{(a_i,b_i),i=1,\dots,n\}$ of $P$, all disks in $\D_{\M}$ have a common intersection, implying that any point $o$ in the common intersection satisfies
\[
	\|a_i-o\|+\|b_i-o\|~\le~ \sqrt{2}~\|a_i-b_i\|, ~\textrm{where } \sqrt{2}\approx 1.4142.
\]
Recently, Barabanshchikova and Polyanskii~\cite{barabanshchikova2022} confirmed the conjecture of Fingerhut.

The statement of Equation~\eqref{eqFingerhut} is equivalent to stating that the intersection $\E(a_1b_1)\cap \E(a_2b_2)\cap\dots\cap \E(a_nb_n)$ is not empty, where $\E(pq)$ is the region of the plane bounded by the ellipse with foci $p$ and $q$, and major axis length $(2/\sqrt{3})~\|p-q\|$ (see~\cite{andy}).

In our context of bichromatic point sets, given $p\in R$ and $q\in B$, let $\E(pq)$ denote the region bounded by the ellipse with foci $p$ and $q$, and major axis length $\sqrt{2}~\|p-q\|$. That is, $\E(pq)=\{x\in\mathbb{R}^2:\|p-x\|+\|q-x\|\le\sqrt{2}~\|p-q\|\}$. Then, the statement of Theorem~\ref{theo:main} is equivalent to stating that the intersection $\E(r_1b_1)\cap \E(r_2b_2)\cap\dots\cap \E(r_nb_n)$ is not empty, for any max-sum matching $\{(r_i,b_i),i=1,2,\ldots,n\}$ of $R\cup B$. 

We note that the factor $\sqrt{2}$ is tight. It suffices to consider two red points and two blue points as vertices of a square, so that each diagonal has vertices of the same color. The center of the square is the only point in common of the two ellipses induced by any max-sum matching. 

Hence, to prove Theorem~\ref{theo:main} it suffices to consider $n\le 3$, by Helly's Theorem. Let $X_1,X_2,\ldots,X_n$ be a collection of $n$ convex subsets of $\mathbb{R}^d$, with $n \ge d + 1$. Helly's Theorem~\cite{helly1923} asserts that if the intersection of every $d + 1$ of these subsets is nonempty, then the whole collection has a nonempty intersection. That is why we prove our claim only for $n\le 3$, since we are considering $n$ ellipses in $\mathbb{R}^2$. The arguments that we give in this paper are a simplification and adaptation of the arguments of Barabanshchikova and Polyanskii~\cite{barabanshchikova2022}.

Huemer et al.~\cite{HUEMER2019} proved that if $\M'$ is any perfect matching of $R$ and $B$ that maximizes the total \emph{squared} Euclidean distance of the matched points, i.e., it maximizes $\sum_{(p,q)\in\M'}\|p-q\|^2$, then all disks of $\D_{\M'}$ have a point in common. As proved by Bereg et al.~\cite{bereg2022}, the disks of our max-sum matching $\M$ of $R\cup B$ intersect pairwise, fact that will be used in this paper, but the common intersection is not always possible.  

\section{Proof of main result}\label{sec-main}

Let $R$ and $B$ be two disjoint point sets defined as above, where $|R|=|B|=n$, $n\le 3$, and let $\M$ be a max-sum matching of $R\cup B$. Note that for every pair $(p,q)\in\M$ the disk $\D(pq)$ is inscribed in the ellipse $\E(pq)$ (see Figure~\ref{fig:disk-ellipse}), which implies $\D(pq)\subset \E(pq)$. Then, for $n=2$ Theorem~\ref{theo:main} is true because the disks of $\M$ intersect pairwise~\cite[Proposition~2.1]{bereg2022}. Trivially, the theorem is also true for $n=1$. Therefore, we will prove in the rest of the paper that the theorem is also true for $n=3$, which will require elaborated arguments.

Let $n=3$, with $R=\{a,b,c\}$ and $B=\{a',b',c'\}$, and let $\M=\{(a,a'),(b,b'),(c,c')\}$ be a max-sum matching of $R\cup B$. 

For two points $p,q\in\mathbb{R}^2$, let $r(pq)$ denote the ray with apex $p$ that goes through $q$, and for a real number $\lambda\ge 1$, let $\E_\lambda(pq)$ be the region bounded by the ellipse with foci $p$ and $q$ and major axis length $\lambda\|p-q\|$. That is,  $\E_\lambda(pq)=\{x\in\mathbb{R}^2:\|p-x\|+\|q-x\|\le \lambda\|p-q\|\}$. Note that in our context $\E(pq)=\E_{\sqrt{2}}(pq)$, and $\E_\lambda(pq)\subset \E_{\lambda'}(pq)$ for any $\lambda'>\lambda$.

Assume by contradiction that $\E(aa')\cap \E(bb')\cap \E(cc')=\emptyset$. Then, we can ``inflate uniformly'' $\E(aa')$, $\E(bb')$, and $\E(cc')$ until they have a common intersection. Formally, we can take the minimum $\lambda>\sqrt{2}$ such that $\E_\lambda(aa')\cap \E_\lambda(bb')\cap \E_\lambda(cc')$ is not empty, case in which $\E_\lambda(aa')\cap \E_\lambda(bb')\cap \E_\lambda(cc')$ is singleton. Let $o$ denote the point of $\E_\lambda(aa')\cap \E_\lambda(bb')\cap \E_\lambda(cc')$.
 
Let $\ell(aa')$ denote the ray with apex $o$ that bisects $r(oa)$ and $r(oa')$. Similarly, we define $\ell(bb')$ and $\ell(cc')$. Let $t(aa')$ denote the line through $o$ tangent to $\E_\lambda(aa')$, oriented so that $\E_\lambda(aa')$ is to its right. Similarly, we define $t(bb')$ and $t(cc')$. It is well known that given an ellipse with foci $p$ and $q$, and a line tangent at it at some point $o$, the rays $r(op)$ and $r(oq)$ form equal angles with the tangent line (see Figure~\ref{fig:tangent-ellipse}). This implies that rays $\ell(aa')$, $\ell(bb')$, and $\ell(cc')$ are perpendicular to the tangent lines $t(aa')$, $t(bb')$, and $t(cc')$, respectively. In other words, they are contained respectively in the normal lines at point $o$.

Since $\E(aa')$, $\E(bb')$, and $\E(cc')$ intersect pairwise (and also none of them is contained inside other one), we have that $o$ belongs to the boundary of each of $\E_\lambda(aa')$, $\E_\lambda(bb')$, and $\E_\lambda(cc')$. Then, $\E_\lambda(aa')$, $\E_\lambda(bb')$, and $\E_\lambda(cc')$ intersect pairwise, and each pairwise intersection contains interior points. This implies that no two lines of $t(aa')$, $t(bb')$, and $t(cc')$ coincide. Furthermore, the six directions (positive and negative) of $t(aa')$, $t(bb')$, and $t(cc')$ alternate around $o$, which implies that any two consecutive rays among $\ell(aa')$, $\ell(bb')$, and $\ell(cc')$ counterclockwise around $o$, have rotation angle strictly less than $\pi$ (see Figure~\ref{fig:rays}).

\begin{figure}[t]
	\centering
	\subfloat[]{
		\includegraphics[scale=1,page=1]{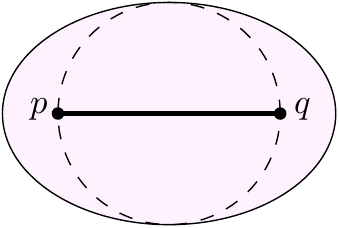}
		\label{fig:disk-ellipse}
	}~~~~
	\subfloat[]{
		\includegraphics[scale=1,page=10]{img.pdf}
		\label{fig:tangent-ellipse}
	}\\ 
	\subfloat[]{
		\includegraphics[scale=1,page=11]{img.pdf}
		\label{fig:rays}
	}
	\caption{\small{
			(a) The ellipse $\E(pq)$ and the disk $\D(pq)$.
			(b) A line tangent to an ellipse forms equal angles with the rays, whose apex is the tangency point, that go through the foci. 
			(c) Point $o$ and the three ellipses.
	}}
	\label{fig:lem:monotone-thing}
\end{figure}

Let $G=(R\cup B,E)$ be the bipartite graph such that $(p,q)\in E$ if and only if $p\in R$, $q\in B$, and either $(p,q)\in\{(a,a'),(b,b'),(c,c')\}$ or $o\in \D(pq)$. We color the edges into two colors: We say that edge $(p,q)$ is {\em black} if $(p,q)\in\{(a,a'),(b,b'),(c,c')\}$. Otherwise, we say that $(p,q)$ is {\em white}. Note that this color classification is consistent, since we have that $o\notin \D(pq)$ for all $(p,q)\in\{(a,a'),(b,b'),(c,c')\}$ because $\D(pq)$ is contained in the interior of $\E_\lambda(pq)$ and $o$ is in the boundary of $\E_\lambda(pq)$.

The proof of the next lemma is included for completeness.

\begin{lemma}[\cite{barabanshchikova2022}]\label{lemma:cycle}
If $G$ has a cycle whose edges are color alternating, then $\M$ is not a max-sum matching of $R\cup B$.  
\end{lemma}

\begin{proof}
For a black edge $(p,q)$ we have that $\|p-o\|+\|q-o\|=\lambda\|p-q\|$. For a white edge $(p,q)$ we have that $\|p-o\|+\|q-o\|<\lambda\|p-q\|$, since $o\in\D(pq)$ and $\D(pq)$ is contained in the interior of $\E_\lambda(pq)$. Let $(r_1,b_1,r_2,b_2,\ldots,r_m,b_m,r_{m+1}=r_1)$ be a color alternating cycle of length $m$, where $r_1,\ldots,r_m\in R$ and $b_1,\ldots,b_m\in B$. Suppose w.l.o.g.\ that the edge $(r_1,b_1)$ is black, which means that the edges $(r_1,b_1),\ldots,(r_m,b_m)\in \M$ are all black, and the edges $(b_1,r_2),\ldots,(b_m,r_{m+1})\in \M$ are all white. Then, we have that:
\[
	\sum_{i=1}^m\|r_i-b_i\| = \frac{1}{\lambda}\sum_{i=1}^m\left(\|r_i-o\|+\|b_i-o\|\right) =
	\frac{1}{\lambda}\sum_{i=1}^m\left(\|b_i-o\|+\|r_{i+1}-o\|\right) <
	\sum_{i=1}^m\|b_i-r_{i+1}\|.
\]
Hence, by replacing in $\M$ the black edges of the cycle by the white edges, we will obtain a matching of larger total sum. 
\end{proof}

\begin{lemma}\label{lemma:white-incident}
Each vertex of $G$ has at least one white edge incident to it.
\end{lemma}

\begin{proof}
Consider the blue vertex $a'$. Assume w.l.o.g.\ that $o$ is the origin of coordinates, and $a'$ is in the positive direction of the $y$-axis. We have that $\angle aoa'<\pi/2$ because $o\notin \D(aa')$, then assume w.l.o.g.\ that $a$ is in the interior of the first quadrant $Q_1$. Let $Q_2$, $Q_3$, and $Q_4$ be the second, third, and fourth quadrants, respectively. Further assume w.l.o.g.\ that rays $\ell(aa')$, $\ell(bb')$, and $\ell(cc')$ appear in this order counterclockwise.

Assume by contradiction that there is no white edge incident to $a'$. This implies that $b,c$ belong to the interior of $Q_1\cup Q_2$. 
If $c\in Q_2$, then the counterclockwise rotation angle from $\ell(cc')$ to $\ell(aa')$ is larger than $\pi$. Hence, $c\in Q_1$. 
If $b\in Q_1$, then the counterclockwise rotation angle from $\ell(aa')$ to $\ell(bb')$, or that from $\ell(bb')$ to $\ell(cc')$, is larger than $\pi$. Hence $b\in Q_2$. 
Furthermore, if both $b'$ and $c'$ belong to $Q_1\cup Q_2$, then the counterclockwise rotation angle from $\ell(bb')$ to $\ell(cc')$ is larger than $\pi$. Hence, at least one of $b',c'$ belong to the interior of $Q_3\cup Q_4$. That is, $b'\in Q_3$ and/or $c'\in Q_4$. The proof is divided now into three cases:

{\bf Case 1:} $b'\in Q_3$ and $c'\in Q_4$. Since $b\in Q_2$ and $c'\in Q_4$, the angle $\angle boc'\ge \pi/2$, which implies that $o\in \D(bc')$ (see Figure~\ref{fig:case1}). That is, edge $(b,c')$ is white. Similarly, edge $(b',c)$ is also white. The colors of the edges of the cycle $(b,c',c,b',b)$ alternate, then Lemma~\ref{lemma:cycle} implies a contradiction.

{\bf Case 2:} $b'\in Q_3$ and $c'\notin Q_4$. Since the counterclockwise rotation angle $\theta$ from $\ell(bb')$ to $\ell(cc')$ is smaller than $\pi$, we must have that $c'\in Q_1$. As in Case~1, we have that edge $(b',c)$ is white, given that $b'\in Q_3$ and $c\in Q_1$. Let $\beta$ be the half of the angle between rays $r(ob)$ and $r(ob')$, and $\gamma$ the half of the angle between the rays $r(oc)$ and $r(oc')$ (see Figure~\ref{fig:case2}). We have that $\beta,\gamma<\pi/4$, which implies that $\angle boc'\ge 2\pi-\beta-\gamma-\theta\ge \pi/2$. Hence, edge $(b,c')$ is also white. Again, the colors of the edges of the cycle $(b,c',c,b',b)$ alternate, and Lemma~\ref{lemma:cycle} implies a contradiction.

{\bf Case 3:} $b'\notin Q_3$ and $c'\in Q_4$. The proof of this case is analogous to that of Case~2.

The lemma thus follows.  
\end{proof}

\begin{figure}[t]
	\centering
	\subfloat[]{
		\includegraphics[scale=1,page=12]{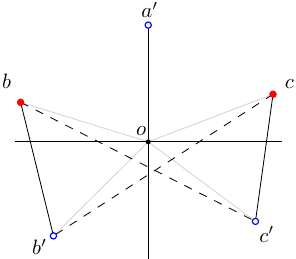}
		\label{fig:case1}
	}~~~~
	\subfloat[]{
		\includegraphics[scale=1,page=13]{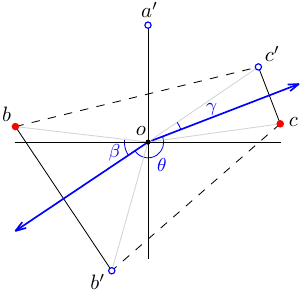}
		\label{fig:case2}
	}
	\caption{\small{
			Proof of Lemma~\ref{lemma:white-incident}. Black edges are in normal line style, and white edges in dashed style.
	}}
	\label{fig:proof-white-incident}
\end{figure}

Lemma~\ref{lemma:white-incident} implies that graph $G$ has always a cycle (of length four or six) whose edges are color alternating. Hence, Lemma~\ref{lemma:cycle} implies a contradiction, and we obtain that the max-sum matching $\M$ ensures that $\E(aa')\cap \E(bb')\cap \E(cc')\neq\emptyset$. Therefore, Theorem~\ref{theo:main} holds.

\small

%{\bf Acknowledgements}: We thank the {\tt GeoGebra} open source software and its developers~\cite{gg}.
%%
%%C.~H.\ was partially supported by projects MTM2015-63791-R (MINECO/FEDER) and Gen.\ Cat.\ DGR 2017SGR1336.
%%
%P.~P-L.\ was partially supported by projects DICYT 041933PL Vicerrector\'ia de Investigaci\'on, Desarrollo e Innovaci\'on USACH (Chile), and Programa Regional STICAMSUD 19-STIC-02.
%%
%C.~S.\ is supported by projects MTM2015-63791-R MINECO/FEDER and Gen.~Cat.\ DGR 2017SGR1640.

\bibliographystyle{abbrv}
\bibliography{refs}

\end{document}